\documentclass[11pt]{amsart}

\usepackage[pdftex]{graphicx}
\usepackage[a4paper,margin=2.5cm]{geometry}

\usepackage{amsfonts}
\usepackage{amsmath}
\usepackage{amssymb}
\usepackage{amsthm}
\usepackage{dsfont}
\usepackage{tikz}
\usepackage{array}

\usepackage{rotating}

\usepackage[T1]{fontenc}

\usepackage{paralist}
\usepackage{color}

\usepackage{overpic}

\usepackage[colorlinks,cite color=blue,pagebackref=true,pdftex]{hyperref}

\renewcommand\int{\mathsf{int}}

\newcommand\Emph[1]{\textbf{#1}}

\DeclareMathOperator{\sgn}{sgn}

\DeclareMathOperator{\Hom}{Hom}
\DeclareMathOperator{\Col}{Col}

\DeclareMathOperator{\Stab}{Stab}

\newtheorem{thm}{Theorem}[section]
\newtheorem{cor}[thm]{Corollary}
\newtheorem{lem}[thm]{Lemma}
\newtheorem{prop}[thm]{Proposition}

\theoremstyle{definition}

\newtheorem{example}[thm]{Example}

\newtheorem{rmk}{Remark}

\title[Order polynomials and P\'{o}lya's enumeration theorem]{Order polynomials and P\'{o}lya's enumeration theorem}

\author{Katharina Jochemko}
\address{Fachbereich Mathematik und Informatik, 
Freie Universit\"at Berlin, 
Germany}
\email{jochemko@math.fu-berlin.de}


\date{\today}
\thanks{The author was supported by a \emph{Hilda Geiringer Scholarship} at
the Berlin Mathematical School. }

\begin{document}
\begin{abstract}
P\'{o}lya's enumeration theorem states that the number of labelings of a finite set up to symmetry is given by a polynomial in the number of labels. We give a new perspective on this theorem by generalizing it to partially ordered sets and order
preserving maps. Further we prove a reciprocity statement in terms of strictly order preserving maps generalizing a classical result by Stanley (1970). We apply our results to counting graph colorings up to symmetry.
\end{abstract}

\maketitle

\section{Introduction}
Counting objects up to symmetry is a basic problem of enumerative combinatorics. A fundamental result in this context is P\'{o}lya's enumeration theorem which is concerned with counting \textbf{labelings} of a set of objects modulo symmetry. Here a labeling of a set $X$ is defined as a map $f\colon X\rightarrow Y$ where $Y$ is the set of labels. If $G$ is a group acting on $X$ then $G$ also acts on the set of labelings $Y^X := \lbrace f\colon X\rightarrow Y \rbrace $. P\'{o}lya's enumeration theorem now states:
\begin{thm}[P\'{o}lya's enumeration theorem \cite{polya1937kombinatorische}]
Let $G$ be a finite group acting on a finite set $X$ and let $Y$ be a finite set of $n=\left|Y\right|$ labels. Then
\[
\left| {Y^X}/G \right| =\frac{1}{|G|}\sum _{g\in G} n^{c(g)}
\]
where ${Y^X}/G$ is the collection of orbits of ${Y^X}$ and $c(g)$ is the number of cycles of $g$ as permutation of $X$.
\end{thm}
We give a new perspective on this theorem by generalizing it in terms of partially ordered sets, or \textbf{posets} for short, and order preserving maps. More precisely, we consider a finite poset $P$ and a group $G$ acting on $P$ by \textbf{automorphisms}. Then $G$ acts in a natural way also on the set of all order preserving maps $\Hom (P,[n])$ from $P$ into the $n$-chain $[n]=\lbrace 1<\cdots <n\rbrace$. We show that the number of \textbf{orbits} of $\Hom (P,[n])$ is given by a polynomial $\Omega _{P,G} (n)$ which we call the \textbf{orbital order polynomial}. P\'{o}lya's enumeration theorem then follows by specializing this result to antichains. Further we give a combinatorial interpretation for $\Omega _{P,G} (-n)$ in terms of orbits of strictly order preserving maps. This naturally generalizes the classical polynomiality and reciprocity theorems for order preserving maps due to Stanley \cite{stanley1970chromatic}. 

The results can be applied to graph colorings. We consider a finite group $G$ acting by automorphisms on a finite simple graph $\Gamma=(V,E)$ and the function $\chi _{\Gamma,G}(n)$ counting proper colorings $c\colon V\rightarrow [n]$ up to group action. Cameron and Kayibi \cite{cameron2007orbital} seem to be the first who considered this function which they called the \textbf{orbital chromatic polynomial}. Previously, Hanlon \cite{hanlon1985chromatic} treated the case of $G$ being the automorphism group of $\Gamma$. It is easy to see that $\chi _{\Gamma,G}(n)$ indeed agrees with a polynomial for all $n\geq 1$. We further give a representation as a sum of order polynomials.  
 
We also give a combinatorial interpretation for evaluating this polynomial at negative integers in terms of acyclic orientations and compatible colorings. This naturally generalizes Stanley's reciprocity for graph colorings \cite{stanley1973acyclic}.

\textbf{Acknowledgments.} We are grateful to Raman Sanyal for suggesting this project and constant support. We also want to thank Alexandru Chirvasitu and Yan X Zhang for fruitful conversations, and Darij Grinberg for careful reading and helpful comments.

\section{Order preserving P\'{o}lya-enumeration}
\subsection{Groups}
Let $G$ be a finite group with identity element $e$ and $X$ be a finite set. A \textbf{group operation} of $G$ on $X$ is a map $\cdot \colon G \times X \rightarrow X$ such that $g\cdot (h\cdot x)=(gh)\cdot x$ and $e\cdot x=x$ for all $g,h\in G$ and $x\in X$. We say that $G$ operates or \textbf{acts} on $X$. For every $g\in G$ we denote by $X^g$ the \textbf{fixpoints} of $g$, i.e.,\ $X^g=\lbrace x \in X \colon g\cdot x=x\rbrace$. For an element $x\in X$ we denote by $Gx=\lbrace g\cdot x \colon g\in G \rbrace $ the \textbf{orbit} of $x$. The set of all orbits partitions $X$ and is called $X/G$.  

Burnside's lemma (see e.g. \cite[Theorem 10.5]{MR1871828}) gives a formula for the number of orbits in terms of fixpoints:
\begin{thm}[Burnside's lemma]\label{thm:burnside}
Let $G$ be a finite group acting on a finite set $X$. Then
\[
\left| X/G \right| \ = \ \frac{1}{|G|}\sum _{g\in G} |X^g|.
\]
\end{thm}
For an element $x\in X$ the \textbf{stabilizer} of $x$ is $\Stab (x)=\lbrace g\in G \colon g\cdot x =x \rbrace$. 
The operation can be restricted to any subgroup $H\subseteq G$. For $g\in G$ we denote by $\langle g\rangle =\lbrace e,g,g^2,\ldots\rbrace$ the cyclic subgroup generated by $g$. The orbit of $x$ under the action of $\langle g \rangle$ is denoted by $[x]_g$ and is called a \textbf{cycle} of $g$. In particular, $c(g)=\left| X/\langle g\rangle\right|$ is the number of cycles of $g$. Identifying $g$ with the corresponding permutation on $X$ gives the usual notion of cycles.
\begin{example}
Let $G=S_n$ be the symmetric group acting on $[n]=\lbrace 1,\ldots ,n\rbrace$ in the usual way. Every permutation $\sigma \in S_n$ can be written as product of disjoint cycles in $S_n$ and this representation is unique up to interchanging the order of the cycles in the product. In this case $[x]_{\sigma} =\lbrace y\in [n]\colon y \text{ and } x \text{ are in the same cycle}\rbrace$ and $c(\sigma)$ is the number of cycles in the unique representation as product of disjoint cycles.
\end{example}
If a set $X$ has additional structure we say that a group $G$ acts by \textbf{automorphisms} on $X$ if the group operation respects the structure, that is, for all $g\in G$ the map $x\mapsto g\cdot x$ is a \textbf{structure preserving bijection}. 

Let $P$ be a finite \textbf{poset}. Then a group $G$ acts on $P$ by automorphisms if for all $g\in G$ and for all $p$ and $q$ in $P$ we have $g\cdot p \prec g \cdot q$ whenever $p\prec q$.

For a finite simple \textbf{graph} $\Gamma=(V,E)$ an action of a group $G$ on $V$ respects the structure of $\Gamma$ if for all edges $uv\in E$ we have that there is an edge between $g\cdot u$ and $g\cdot v$ for all $g\in G$.

The operation of $G$ on $X$ induces an operation on $Y^X$. This \textbf{induced operation} is defined by $(g \cdot f)(x)=f(g^{-1} x)$ for all $g \in G$, $f\in Y^X$ and $x \in X$.
\subsection{Order preserving maps}
Let $P$ be a finite poset. A map $\phi \colon P\rightarrow [n]$ is called \textbf{order preserving} if $\phi (p)\leq \phi (q)$ whenever $p\prec q$ and equality is forbidden for \textbf{strictly} order preserving maps. We denote the set of all order preserving maps by $\Hom (P,[n])$, and the set of strictly order preserving maps by $\Hom ^{\circ} (P,[n])$. Their cardinalities are given by the \textbf{order polynomials} $\Omega _P (n)$ and $\Omega ^{\circ} _P (n)$ respectively. The following classical result is due to Stanley \cite{stanley1970chromatic}.
\begin{thm}[Stanley \cite{stanley1970chromatic}]\label{thm:stan_ord}
For a finite poset $P$ the function $\Omega _P (n)$ agrees with a polynomial of degree $\left|P\right|$ for all $n\geq 1$, and
\[
\Omega ^{\circ} _P (n)=(-1)^{\left|P\right|}\Omega _P (-n).
\]
\end{thm}
 For every finite group $G$ acting on $P$ by automorphisms and for every $g\in G$ we define a partial order on $P_g=P/\langle g\rangle$ by defining $[x]_g \prec [y]_g$ whenever there are   $\tilde{x} \in [x]_g$ and $\tilde{y} \in [y]_g$ such that $\tilde{x}\prec \tilde{y}$.\\
This, in fact, yields a poset, the \textbf{quotient poset} (see e.g. \cite{stanley1984quotients}).
\begin{lem}
Let $P$ be a finite poset and $G$ a finite group acting by automorphisms on $P$. Then $P_g$ is a poset with $c(g)$ elements for all $g\in G$.
\end{lem}
\begin{proof}
First we observe that every cycle $[x]_g$ is an antichain. For that suppose $g^l x\prec g^k x$ for some $k,l \geq 0$ and consequently $x\prec g^{k-l}x$. From that we get $x\prec g^{k-l}x \prec g^{2(k-l)}x\prec \cdots \prec x$ as $g^{k-l}$ has finite order which is a contradiction.

For antisymmetry assume there are $\tilde{x},\bar{x} \in [x]_g$ and $\tilde{y},\bar{y} \in [y]_g$ with $\tilde{x}\prec \tilde{y}$ and $\bar{y}\prec \bar{x}$. Then $\bar{y}=g^l\tilde{y}$ for some $l\geq 0$ and we have $g^l \tilde{x}\prec g^l \tilde{y}=\bar{y}\prec \bar{x}$ which contradicts $[x]_g$ being an antichain.

For transitivity we let $\tilde{x}\in [x]_g$, $\tilde{y},\bar{y} \in [y]_g$ and $\bar{z}\in [z]_g$ with $\tilde{x} \prec \tilde{y}$ and $\bar{y} \prec \bar{z}$. Then there exists a $k\geq 0$ with $g^k \tilde{y}=\bar{y}$ and we have $g^k \tilde{x} \prec g^k \tilde{y}=\bar{y}\prec \bar{z}$.
\end{proof}
An order preserving action of $G$ on $P$ induces an action on $\Hom (P,[n])$ and $\Hom ^{\circ} (P,[n])$ as subsets of $[n]^P$.  For $n\geq 1$ we define the \textbf{orbital order polynomials} $\Omega _{P,G}(n)=\left| \Hom (P,[n])/G \right|$  and $\Omega _{P,G}^{\circ}(n)=\left| \Hom ^{\circ}(P,[n])/G \right|$. The following main theorem states that $\Omega _{P,G}(n)$ and $\Omega ^{\circ}_{P,G}(n)$ are indeed polynomials for $n\geq 1$, and gives formulas in terms of order polynomials:
\begin{thm}\label{thm:poly_poset}
Let $G$ be a finite group acting by automorphisms on a finite poset $P$. Then 
\begin{eqnarray}
\Omega _{P,G}(n) &=&\frac{1}{|G|}\sum _{g\in G} \Omega _{P_g} (n),\label{eqn:poly_poset}\\
\Omega ^{\circ}_{P,G}(n) &=&\frac{1}{|G|}\sum _{g\in G} \Omega ^{\circ}_{P_g} (n)\label{eqn:poly_posetstrict}
\end{eqnarray}
for $n\geq 1$. In particular, $\Omega _{P,G}(n)$ and $\Omega ^{\circ}_{P,G}(n)$ agree with polynomials of degree $|P|$ for $n\geq 1$.
\end{thm}

\begin{proof}
We only show equation \eqref{eqn:poly_poset} as the argument for equation \eqref{eqn:poly_posetstrict} is analoguous. By Theorem \ref{thm:burnside} we have 
\[
\left| \Hom (P,[n])/G \right| =\frac{1}{|G|}\sum _{g\in G} \left|\Hom (P,[n]) ^g\right|
\]
By definition, $\phi \in \Hom (P,[n])^g$ if and only if $\phi (g^{-1} x)=\phi (x)$ for all $x\in P$. But this is equivalent to $\phi$ being constant on $[x]_g$. Therefore 
\begin{eqnarray*}
\Hom (P,[n])^g &\rightarrow & \Hom (P_g,[n])\\
\varphi &\mapsto &([x]_g \mapsto \varphi (x))
\end{eqnarray*}
is a one-to-one correspondence. Further, observe that  $\deg \Omega _{P_g} (n)=c(g)=|P|$ if and only if $g$ acts trivially on $P$, and $c(e)=|P|$.
\end{proof}

By applying Theorem \ref{thm:poly_poset} to antichains, we get P\'{o}lya's enumeration theorem in the language of posets:
\begin{cor}\label{cor:P\'{o}lya}
Let $G$ be a finite group acting by automorphisms on a finite antichain $A$, and let $Y=[n]$. Then 
\[
\left| {Y^A}/G \right| =\frac{1}{|G|}\sum _{g\in G} n^{c(g)}.
\]
\end{cor}
\begin{proof}
As $A$ is an antichain we have $\Hom (A,[n])=Y^A$. The result follows by observing that $A_g$ is an antichain with $c(g)$ elements for all $g\in G$ and $\Omega _{A_g}(n)=n^{\left|A_g\right|}$.
\end{proof}
\begin{example}\label{ex:symm_group}
Let $S_k$ be the symmetric group acting on an antichain $A=\lbrace x_1,\ldots ,x_k\rbrace$ on $k$ elements by permuting indices, and let $Y=[n]$. Then
\[
\lbrace \phi \in Y^A \colon 1\leq \phi (x_1)\leq \cdots \leq \phi(x_k)\leq n\rbrace 
\]
is a set of representatives of $\Hom(A,Y)/{S_k}$ and therefore $\Omega _{A,S_k}(n) =\Omega _{[k]}(n)$. By Corollary \ref{cor:P\'{o}lya} we conclude that
\[
\binom {n+k-1}{k}=\frac{1}{k!}\sum _{\sigma \in S_k} n^{c(\sigma)}.
\]
\end{example}
\subsection{Combinatorial reciprocity}
Let $G$ be a finite group acting on a finite poset $P$ by automorphisms. As the number of orbits $\left| \Hom (P,[n])/G \right|$ agrees with a polynomial by Theorem \ref{thm:poly_poset} we ponder the question if there is a combinatorial interpretation for evaluating this polynomial at negative integers. For that we have to consider a certain class of order preserving maps. The \textbf{sign} $\sgn (g)$ of an element $g\in G$ is defined as the sign of $g$ as a permutation of $P$ and is equal to $(-1)^{|P|+c(g)}$. An order preserving map $\phi \in \Hom (P,[n])$ is called \textbf{even} if for all $g\in \Stab (\phi)$ we have $\sgn (g)=1$. The set of all even order preserving maps is denoted by $\Hom _{+} (P,[n])$, and we define $\Hom ^{\circ}_{+} (P,[n]):=\Hom _{+} (P,[n])\cap \Hom ^{\circ} (P,[n])$ to be the set of even strictly order preserving maps. One observes that the action of $G$ on $\Hom (P,[n])$ restricts to an action on $\Hom_{+} (P,[n])$. For these notions the following reciprocities hold: 
\begin{thm}\label{thm:reciprocity}
Let $G$ be a finite group acting by automorphisms on a finite poset $P$. Then 
\begin{eqnarray}
\Omega _{P,G}(-n)&=&(-1)^{|P|} \left| \Hom ^{\circ} _{+} (P,[n])/G \right|,\label{eq:recip}\\
\Omega ^{\circ}_{P,G}(-n)&=&(-1)^{|P|} \left| \Hom _{+} (P,[n])/G \right|.\label{eq:recipstrict}
\end{eqnarray}
\end{thm}
\begin{proof}
Again, we only show equation \eqref{eq:recip} as equation \eqref{eq:recipstrict} follows by analogous arguments.
By equation \eqref{eqn:poly_poset} and Theorem \ref{thm:stan_ord} we have
\begin{equation}\label{eqn:poly_poset2}
\Omega _{P,G}(-n)=\frac{1}{|G|}\sum _{g\in G} (-1)^{|P_g|}\Omega ^{\circ} _{P_g} (n).
\end{equation}
We observe that $|P_g|=c(g)$ is the number of orbits under the action of $\langle g\rangle$. Therefore equation \eqref{eqn:poly_poset2} becomes
\begin{equation*}
\Omega _{P,G}(-n)=(-1)^{|P|}\frac{1}{|G|}\sum _{g\in G} \sgn (g)\left| \Hom^{\circ} (P,[n])^g\right|=(-1)^{|P|}\frac{1}{|G|}\sum _{\phi \in \Hom^{\circ} (P,[n])} \sum _{g\in \Stab (\phi)} \sgn (g).
\end{equation*}
For $\phi \in \Hom ^{\circ} (P,[n])$ and $g_0 \in \Stab (\varphi)$ such that $\sgn (g_0)=-1$ there is a bijection
\begin{eqnarray*}
\lbrace g\in \Stab (\phi) \colon \sgn (g)=1 \rbrace & \longrightarrow & \lbrace g\in \Stab (\phi) \colon \sgn (g)=-1 \rbrace\\
g& \mapsto &g_0 g
\end{eqnarray*} 
Hence, $\sum _{g\in \Stab (\phi)} \sgn (g)=0$ whenever $\phi$ is not even.
Therefore the right-hand side of equation \eqref{eqn:poly_poset2} equals
\[
(-1)^{|P|}\frac{1}{|G|}\sum _{g\in G}\left| \Hom^{\circ} _+(P,[n])^g\right|
\]
which equals $(-1)^{|P|}\left| \Hom ^{\circ} _{+} (P,[n])/G \right|$ by Theorem \ref{thm:burnside}. 
\end{proof}
In the setting of P\'{o}lya's enumeration theorem the statement simplifies:
\begin{cor}
Let $G$ be a finite group acting on a finite antichain $A$, and let $Y=[n]$. Then 
\[
\Omega _{A,G}(-n) =(-1)^{|A|} \left| \Hom _{+} (A,[n])/G \right|
\]
\end{cor}
\begin{proof}
This follows from the fact that every order preserving map from an antichain is automatically strictly order preserving.
\end{proof}
\begin{example}
In Example \ref{ex:symm_group} we have $\phi \in \Hom _{+} (A,[n])$ if and only if $\phi$ is injective. Therefore 
\[
\lbrace \phi \in Y^A \colon 1\leq \phi (x_1)< \cdots < \phi(x_k)\leq n\rbrace 
\]
is a set of representatives for $\Hom _{+} (A,[n])/S_k$. This is reflected by the fact that $(-1)^k\Omega _{[k]}(-n)$ equals by  Theorem \ref{thm:stan_ord} the number of strictly order preserving maps from $[k]$ to $[n]$.
\end{example}
\begin{rmk}
An alternative, geometric route is by way of Ehrhart theory of order polytopes. Geometrically the setting can be translated into counting lattice points in order polytopes where the action of the group is given by permuting coordinates. This complements results by Stapledon \cite{stapledon2011equivariant} who considers lattice preserving group actions and counts lattice points inside stable rational polytopes.
\end{rmk}
\section{Graphs}
Let $\Gamma=(V,E)$ be a finite simple graph and let $G$ be a finite group acting on $\Gamma$ by automorphisms. A \textbf{$n$-coloring} of $\Gamma$ is a map $c\colon V \rightarrow [n]$. The coloring is called \textbf{proper} if $c(v)\neq c(w)$ whenever there is an edge between $v$ and $w$. The action of $G$ on $\Gamma$ induces an action on the set of all colorings, and also on the set of all proper colorings which we denote by $\Col _n (\Gamma)$. The \textbf{orbital chromatic polynomial} $\chi _{\Gamma,G}$ is defined by $\chi _{\Gamma,G}(n)=\left| \Col _n (\Gamma)/G\right|$ for all $n\geq 1$. An \Emph{orientation} $\sigma\colon E\rightarrow V$ of $\Gamma$ assigns to every edge $e$ a vertex of $e$ called its head. An orientation is \Emph{acyclic} if there are no directed cycles. Every acyclic orientation $\sigma$ induces a partial order on the vertex set of $\Gamma$ by defining $v\prec _{\sigma}w$ if there is a directed path from $v$ to $w$. For the corresponding poset we write $\Gamma ^{\sigma}$. $G$ acts on the set $\Sigma$ of all acyclic orientations of $\Gamma$: For an edge $uv$ we define $(g\cdot \sigma )(uv)=g\cdot\sigma (g^{-1} \cdot uv)$.
The next theorem gives us an expression of $\chi _{\Gamma,G}(n)$ in terms of order polynomials. In particular, $\chi _{\Gamma,G}(n)$ is a polynomial for all $n\geq 1$.
\begin{thm}\label{thm:poly_graph}
Let $\Gamma$ be a graph and let $G$ be a group acting on $\Gamma$. Then $G$ acts on $\Col _n (\Gamma)$ and we have 
\[
\chi _{\Gamma,G}(n) =\frac{1}{|G|} \sum _{g\in G} \sum _{\sigma \in \Sigma ^g}  \Omega ^{\circ} _{\Gamma^{\sigma}_g} (n)
\]
for all $n\geq 1$. In particular, $\chi _{\Gamma,G}(n)$ agrees with a polynomial of degree $|\Gamma|$ for all $n\geq 1$.
\end{thm}

\begin{proof}
By Theorem \ref{thm:burnside} we have 
\[
 \left| \Col _n (\Gamma)/G\right| =\frac{1}{|G|}\sum _{g\in G} \left| \Col _n(\Gamma)^g\right|
\]
Let $\phi$ be an element of $\Col _n(\Gamma)^g$ and let $\sigma$ be the acyclic orientation induced by the coloring $\phi$, i.e.\ an edge $e = uv$ is oriented from $u$ to $v$ whenever $\phi(u) < \phi(v)$. Then $\phi$ is a strictly order preserving map from $\Gamma^{\sigma}$ into $[n]$ and $\sigma$ is fixed by $g$, because for every edge $vw\in E$ we have $v\prec _{\sigma} w$ by definition if and only if $\phi (v)< \phi (w)$, and as $\phi$ is fixed by $g$, this implies $\phi (gv)< \phi (gw)$ which is equivalent to $gv\prec _{\sigma} gw$, i.e. $\sigma \in \Sigma ^g$. 
\end{proof}

\begin{example}
Let $k>2$. We consider a cycle $C^k$ on $k$ vertices $\lbrace x_i\rbrace _{i\in \mathbb{Z}_k}$. Then its symmetry group is the dihedral group
\[
D_k \ = \ \langle r,s \mid r^k=1 , s^2 =1 , srs^{-1}=r^{-1}\rangle
\]
which acts on $C^k$ by
\begin{eqnarray*}
r \cdot x_i &=&x_{i+1}\\
s\cdot x_i &=& x_{-i}
\end{eqnarray*}
Then
\[
\chi _{C^k,D_k}(n)=\frac{1}{2k}\left(\sum _{l=1}^k \left|  \Col _n(C^k)^{r^l} \right| +\sum _{l=1}^k \left|  \Col _n(C^k)^{s r^l} \right|\right).
\]
Let $c\in \Col _n(C^k)$. If $l=2q$ is even, then $sr^{2q} \cdot c =c \Leftrightarrow (r^q \cdot c)=s\cdot(r^q \cdot c)$ and therefore
\[
\left|  \Col _n(C^k)^{sr^l} \right|=\left|  \Col _n(C^k)^{s} \right|=\begin{cases} 
\left|  \Col _n([\frac{k}{2}+1]) \right| & \text{ if }k \text{ is even,}\\
0 & \text{ otherwise.}
\end{cases}
\]
If $l=2q+1$ is odd, then $sr^{2q+1} \cdot c =c \Leftrightarrow (r^{q+1} \cdot c)=rs\cdot (r^{q+1} \cdot c)$ and therefore
\[
\left|  \Col _n(C^k)^{s\cdot r^l} \right|=\left|  \Col _n(C^k)^{rs} \right|=0
\]
as $rs\cdot x_0=x_1$, and $x_0$ and $x_1$ are connected by an edge.
Further, for all $1\leq l \leq k$ we obtain
\[
\left|  \Col _n(C^k)^{r^l} \right|=\begin{cases} 
\left|  \Col _n(C^{m}) \right| & \text{ if }m=\gcd (l,k)\neq 1,\\
0 & \text{ otherwise.} 
\end{cases}
\]
If $k>1$ is odd, we therefore get 
\[
\chi _{C^k,D_k}(n)=\frac{1}{2}\chi _{C^k, \mathbb{Z}_k} (n),
\]
with $\mathbb{Z}_k:=\mathbb{Z}/k\mathbb{Z}=\langle r \rangle \subset D_k$. If $k=p> 2$ is a prime this simplifies even further:
\[
\left| \Col _n(C^p)/D_p \right| =\frac{1}{2p} \left|  \Col _n(C^{p}) \right|.
\]
This example is reminiscient of counting necklaces with colored beads (see e.g. \cite[Chapter 35]{MR1871828}).
\end{example}
A pair $(c,\sigma)$ consisting of a coloring $c\colon V\rightarrow [n]$ and an acyclic orientation $\sigma \colon E \rightarrow V$ is called \textbf{weakly compatible} if for every edge $e = uv$ we have $\sigma (uv)=v$ whenever $c(u) < c(v)$.
We define 
\[
\Sigma _n (\Gamma)=\{ (c,\sigma)\in [n]^V \times \Sigma \colon  \text{weakly compatible}\}
\]
We observe that if $G$ acts on $\Gamma$ by automorphisms it also acts on $\Sigma _n (\Gamma)$ by $g\cdot (c,\sigma)=(g\cdot c, g\cdot \sigma)$ for all $(c,\sigma)\in\Sigma _n (\Gamma)$ and $g\in G$. An element $(c,\sigma)\in \Sigma _n (\Gamma)$ is called \textbf{even} if for all $g\in \Stab((c,\sigma))$ we have $\sgn (g)=1$ as permutation of the vertices. We denote the set of all even elements of $\Sigma _{n} (\Gamma)$ by $\Sigma _{n,+} (\Gamma)$. The action of $G$ restricts to an action on $\Sigma _{n,+} (\Gamma)$. We get the following reciprocity statement:

\begin{thm}\label{thm:reciprocity2}
Let $\Gamma$ be a graph and $G$ a group acting on $\Gamma$. Then 
\[
\chi _{\Gamma, G}(-n)=(-1)^{|\Gamma|} \left|\Sigma _{n,+} (\Gamma)/G\right|
\]
\end{thm}
\begin{proof}
By Theorem \ref{thm:poly_graph} and Theorem \ref{thm:stan_ord} and \mbox{$\sgn (g) =(-1)^{|\Gamma|+c(g)}$} we have
\[
\chi _{\Gamma, G}(-n)=(-1)^{\left|\Gamma\right|}\frac{1}{|G|}\sum _{g\in G}\sgn (g)\sum _{\sigma \in \Sigma ^g} \Omega _{\Gamma^{\sigma}_g} (n).
\]
As in the proof of Theorem \ref{thm:poly_poset} we see $\Omega _{\Gamma^{\sigma}_g} (n)= |\Hom (\Gamma ^{\sigma} , [n]) ^g |$, and we observe
\begin{equation}\label{eqn:graphrec}
\left| \Sigma _n (\Gamma)^g\right| =\sum _{\sigma \in \Sigma ^g} |\Hom (\Gamma ^{\sigma} , [n]) ^g |.
\end{equation}
Now we argue the same way as in the proof of Theorem \ref{thm:reciprocity}: By equation \eqref{eqn:graphrec} we get
\begin{equation}\label{eqn:graphrec2}
\chi _{\Gamma,G}(-n)=(-1)^{\left|\Gamma\right|}\frac{1}{|G|}\sum _{g\in G}\sgn g\left| \Sigma _n(\Gamma)^g\right|=(-1)^{\left|\Gamma\right|}\frac{1}{|G|}\sum _{(c,\sigma)\in\Sigma _n(\Gamma)}\sum _{g\in \Stab ((c,\sigma))} \sgn (g) .
\end{equation}
For $(c,\sigma) \in \Sigma _n(\Gamma)$ and $g_0 \in \Stab((c,\sigma))$ such that  $\sgn g_0=-1$ as permutation of the vertices there is a bijection
\begin{eqnarray*}
\lbrace g\in \Stab ((c,\sigma)) \colon \sgn (g)=1 \rbrace & \longrightarrow & \lbrace g\in \Stab ((c,\sigma)) \colon \sgn (g)=-1 \rbrace\\
g& \mapsto &g_0 g
\end{eqnarray*} 
Hence, $\sum _{g\in \Stab ((c,\sigma))} \sgn (g)=0$ whenever $(c,\sigma)$ is not even.
Therefore the right hand side of equation \eqref{eqn:graphrec2} equals
\[
(-1)^{|\Gamma|}\frac{1}{|G|}\sum _{g\in G}\left| \Sigma _{n,+}(\Gamma)^g\right|
\]
which by Theorem \ref{thm:burnside} equals $(-1)^{|\Gamma|}\left| \Sigma_{n,+}(\Gamma)/G \right|$.
\end{proof}

An easy interpretation can be given in the case of $G=\mathbb{Z}_2$:
\begin{cor}
Let $\Gamma$ be a graph and let $\mathbb{Z}_2=\lbrace e,\tau\rbrace$ act on $\Gamma$ by automorphisms such that $\sgn \tau=-1$. Then
\[
\chi _{\Gamma,\mathbb{Z}_2 }(-1)=(-1)^{| \Gamma|} \frac{\left| \Sigma _+ \right|}{2}
\]
where $\Sigma _+=\Sigma _{1,+}(\Gamma)$ is the set of even acyclic orientations of $\Gamma$.
\end{cor}
For $G$ acting trivially on $\Gamma$ we recover a well-known theorem by Stanley:
\begin{cor}[{\cite[Thm.~1.2]{stanley1973acyclic}}] 
    Let $\Gamma$ be a graph and $\chi_{\Gamma}$ its chromatic polynomial. Then $|\chi_{\Gamma}(-n)|$ equals the number of weakly compatible pairs $(c,\sigma)$
    consisting of a $n$-coloring $c$ and an acyclic orientation $\sigma$ of $\Gamma$. In particular, $|\chi_{\Gamma}(-1)|$ is the number of
    acyclic orientations of $\Gamma$.
\end{cor}

Similarly as in Theorem \ref{thm:reciprocity} there is a twin reciprocity in the case of graph colorings. We say that a $n$-coloring $c$ of $\Gamma$ is \textbf{even} if for all $g\in \Stab (c)$ we have $\sgn g=1$ and define $\Col _{n,+} (\Gamma)$ as the set of all even proper $n$-colorings of $\Gamma$. Then the action of $G$ on $\Col _{n} (\Gamma)$ restricts to an action on $\Col _{n,+} (\Gamma)$. We further define $\chi ^+ _{\Gamma, G}(n)=\left|\Col _{n,+} (\Gamma)/G\right|$ as the function counting the number of orbits of even proper $n$-colorings for $n\geq 1$. By similar arguments as in Theorem \ref{thm:poly_graph} and Theorem \ref{thm:reciprocity2} we then have the following:
\begin{prop}
Let $\Gamma$ be a graph and $G$ a group acting on $\Gamma$ by automorphisms. Then $\chi ^+ _{\Gamma,G}(n)$ agrees with a polynomial of degree $|\Gamma|$ for $n\geq 1$ and we have
\[
\chi ^+ _{\Gamma, G}(-n)=(-1)^{|\Gamma|} \left|\Sigma _{n} (\Gamma)/G\right|.
\]
\end{prop}

\nocite{*}
\bibliographystyle{siam}
\bibliography{Polya2}

\begin{thebibliography}{1}

\bibitem{cameron2007orbital}
{\sc P.~J. Cameron and K.~K. Kayibi}, {\em Orbital chromatic and flow roots},
  Combinatorics, Probability and Computing, 16 (2007), pp.~401--407.

\bibitem{hanlon1985chromatic}
{\sc P.~Hanlon}, {\em The chromatic polynomial of an unlabeled graph}, Journal
  of Combinatorial Theory, Series B, 38 (1985), pp.~226--239.

\bibitem{polya1937kombinatorische}
{\sc G.~P{\'o}lya}, {\em Kombinatorische {A}nzahlbestimmungen f{\"u}r
  {G}ruppen, {G}raphen und chemische {V}erbindungen}, Acta mathematica, 68
  (1937), pp.~145--254.

\bibitem{stanley1970chromatic}
{\sc R.~P. Stanley}, {\em A chromatic-like polynomial for ordered sets}, in
  Proceedings of the Second Chapel Hill Conference on Combinatorial Mathematics
  and its Applications, 1970, pp.~421--427.

\bibitem{stanley1973acyclic}
\leavevmode\vrule height 2pt depth -1.6pt width 23pt, {\em Acyclic orientations
  of graphs}, Discrete Mathematics, 5 (1973), pp.~171--178.

\bibitem{stanley1984quotients}
\leavevmode\vrule height 2pt depth -1.6pt width 23pt, {\em Quotients of {P}eck
  posets}, Order, 1 (1984), pp.~29--34.

\bibitem{stapledon2011equivariant}
{\sc A.~Stapledon}, {\em Equivariant {E}hrhart theory}, Advances in
  Mathematics, 226 (2011), pp.~3622--3654.

\bibitem{MR1871828}
{\sc J.~H. van Lint and R.~M. Wilson}, {\em A course in combinatorics},
  Cambridge University Press, Cambridge, second~ed., 2001.

\end{thebibliography}
\end{document}